\pdfoutput=1
\RequirePackage{ifpdf}
\ifpdf 
\documentclass[pdftex]{sigma}
\else
\documentclass{sigma}
\fi

\numberwithin{equation}{section}

\newtheorem{Theorem}{Theorem}[section]

\newtheorem{Lemma}[Theorem]{Lemma}

\newtheorem{cex}[Theorem]{Counterexample}
{\theoremstyle{definition}
\newtheorem{Remark}[Theorem]{Remark} }

\def\a{\alpha}

\newcommand{\esf}{\Omega_{2q}}
\newcommand{\esfqq}{\Omega_{2q+2}}
\newcommand{\ccir}{\Omega_{2}}
\newcommand{\D}{{\cal D}}
\newcommand{\DD}{{\mathbb D}}
\newcommand{\Dz}{{\cal D}_{z}}
\newcommand{\Dcz}{{\cal D}_{\ovl{z}}}
\newcommand{\I}{{\cal I}}
\newcommand{\Ic}{{\overline{\cal I}}}
\def\la{\langle}
\def\ra{\rangle}
\def\ovl{\overline}

\begin{document}
\allowdisplaybreaks

\newcommand{\arXivNumber}{1704.01237}

\renewcommand{\PaperNumber}{088}

\FirstPageHeading

\ShortArticleName{Positive Def\/inite Functions on Complex Spheres and their Walks through Dimensions}

\ArticleName{Positive Def\/inite Functions on Complex Spheres\\ and their Walks through Dimensions}

\Author{Eugenio MASSA~$^\dag$, Ana Paula PERON~$^\dag$ and Emilio PORCU~$^{\ddag\S}$}

\AuthorNameForHeading{E.~Massa, A.P.~Peron and E.~Porcu}

\Address{$^\dag$~Departamento de Matem\'{a}tica, ICMC-USP - S\~{a}o Carlos,\\
\hphantom{$^\dag$}~Caixa Postal 668, 13560-970 S\~{a}o Carlos SP, Brazil}
\EmailD{\href{mailto:eug.massa@gmail.com}{eug.massa@gmail.com}, \href{mailto:apperon@icmc.usp.br}{apperon@icmc.usp.br}}

\Address{$^\ddag$~School of Mathematics and Statistics, Chair of Spatial Analytics Methods,\\
\hphantom{$^\ddag$}~University of Newcastle, UK}
\EmailD{\href{mailto:emilio.porcu@newcastle.ac.edu}{emilio.porcu@newcastle.ac.edu}}

\Address{$^\S$~Department of Mathematics, Universidad T\'ecnica Federico Santa Maria,\\
\hphantom{$^\S$}~Avenida Espa\~na 1680, Valpara\'{\i}so, 230123, Chile}

\ArticleDates{Received April 06, 2017, in f\/inal form October 30, 2017; Published online November 08, 2017}

\Abstract{We provide walks through dimensions for isotropic positive def\/inite functions def\/ined over complex spheres. We show that the analogues of Mont{\'e}e and Descente operators as proposed by Beatson and zu Castell [\textit{J.~Approx. Theory} \textbf{221} (2017), 22--37] on the basis of the original Matheron operator [Les variables r\'egionalis\'ees et leur estimation, Masson, Paris, 1965], allow for similar walks through dimensions. We show that the Mont{\'e}e operators also preserve, up to a constant, strict positive def\/initeness. For the Descente operators, we show that strict positive def\/initeness is preserved under some additional conditions, but we provide counterexamples showing that this is not true in general. We also provide a list of parametric families of (strictly) positive def\/inite functions over complex spheres, which are important for several applications.}


\Classification{42A82; 42C10; 42C05; 30E10; 62M30}

\section{Introduction and main results}\label{s-introd}

Positive def\/inite functions have a long history which can be traced back to papers by Ca\-ra\-th\'eo\-do\-ry, Herglotz, Bernstein and Matthias, culminating in Bochner's theorem from 1932--1933. See Berg \cite{berg-cap-book} for details. In the last twenty years several results related to this topic were obtained in f\/ields as diverse as mathematical analysis, numerical analysis, potential theory, probability theory and geostatistics: we refer the reader to the surveys in Schaback \cite{schaback-survey99,schaback-2000}, Berg \cite{berg-cap-book} and Fasshauer \cite{fassh-2011} for a complete list of references in this direction.

Positive def\/inite radial functions have been known since the two seminal papers by Schoen\-berg~\cite{scho38, scho-42}. The former is devoted to radially symmetric functions depending on the Euclidean distance, and the latter to isotropic functions on unit spheres ${\mathbb S}^d$ of $\mathbb{R}^{d+1}$. Literature on radially symmetric functions on Euclidean spaces has been especially fervent. In his essay devoted to the {\em clavier spherique}, Matheron \cite{matheron} proposed operators called {\em Mont{\'e}e} and {\em Descente} that preserve the property of positive def\/initeness but changing the dimension of the space initially considered. Such a property has been called {\em walk through dimensions}. It is worth noting that the walk through dimensions is achieved at the expense of modifying the dif\/ferentiability at the origin of a given candidate function. Wendland \cite{wendland1995} used the Mont{\'e}e operator with a class of compactly supported radial basis functions, termed Wendland's functions after his works. Schaback \cite{Schaback2011} covered the missing cases of walks through dimensions. Porcu et al.~\cite{zaz-bevi} used a~fractional version of the Mont{\'e}e operator to obtain generalized versions of Wendland's functions. For a reference on walks through dimensions in the geostatistical setting, the reader is referred to Gneiting \cite{gn} and to the more recent work of Porcu and Zastavnyi \cite{porcu-zas2014}.

Positive def\/inite functions as well as strictly positive def\/inite functions in several contexts have been deeply studied by the mathematical analysis literature, and the reader is referred to the works by Menegatto et al.\ (see Chen et al.~\cite{men-chen-sun}, Menegatto and Peron \cite{P-valdir-pd-esfcompl}, Guella et al.~\cite{P-jean-men-pdSMxSm}, and references therein). The use of positive def\/inite functions on real spheres for geostatisticians has arrived recently, thanks to the survey by Gneiting \cite{gneiting-2013} and the recent developments by Berg and Porcu \cite{porcu-berg} and Porcu et al. \cite{porcu-bev-gent}. In particular, Berg and Porcu \cite{porcu-berg} characterized the class of the positive def\/inite functions on the product of ${\mathbb S}^d$ with a locally compact group, extending the Schoenberg's class $\Psi_d$ of the positive def\/inite functions on ${\mathbb S}^d$ (Schoenberg \cite{scho-42}).

A continuous function $f\colon [-1,1]\to\mathbb R$ belongs to the class $\Psi_d$ when the kernel \begin{gather*}K\colon {\mathbb S}^d \times {\mathbb S}^d \to \mathbb{R} \colon K(\xi,\eta) = f(\la\xi,\eta\ra)\end{gather*} is positive def\/inite. Schoenberg \cite{scho-42} proved that $f\in\Psi_d$ if, and only if,
\begin{gather}\label{eq-exp-scho}
f(x)=\sum_{k\geq0}a_k^d c_k(d,x), \qquad \sum_{k\geq0}a_k^d<\infty, \qquad a_k^d\geq0, \quad \forall\, k\geq0,
\end{gather}
where $c_k(d,\cdot)$ are the normalized Gegenbauer polynomials associated to the index $d$ (see Szeg\H{o} \cite[p.~80]{szego}). The coef\/f\/icients in the above series are called {\em $d$-Schoenberg coefficients}. On the other hand, the subclass $\Psi_d^+$ of $\Psi_d$ of the strict positive def\/inite functions on ${\mathbb S}^d$, $d\geq2$, was characterized by Chen et al.~\cite{men-chen-sun}: $f\in\Psi_d^+$ if, and only if, the set $\{k\colon a_k^d>0\}$ contains inf\/initely many odd and inf\/initely many even integers.

The class $\Psi_d$ has received special interest in the last twenty years, while walks through dimensions for positive def\/inite functions on real spheres have been studied in the recent tour de force by Beatson and zu Castell \cite{beat-zucatell-2016, beat-zucatell}. In particular, Beatson and zu Castell \cite{beat-zucatell} def\/ine the {\em Mont{\'e}e} operator
\begin{gather*}
 (If)(x) = \int_{-1}^x f(u){\rm d}u, \qquad x\in[-1,1],
\end{gather*}
for $f$ integrable in $[-1,1]$, and the {\em Descente} operator
\begin{gather*}
(Df)(x) = \frac{{\rm d}}{{\rm d}x}f(x), \qquad x\in[-1,1],
\end{gather*} for $f$ absolutely continuous in $[-1,1]$.
They prove that, for $d\geq2$:
	\begin{enumerate}\itemsep=0pt
	\item[(i)] if $f\in\Psi_{d+2}$, then there exists a constant $c$ such that $c+I f\in \Psi_d$;
	\item[(ii)] if $f\in\Psi_{d+2}^+$, then there exists a constant $c$ such that $c+I f\in\Psi_d^+$;
	\item[(iii)] if $f\in\Psi_{d+2}$, $f\geq0$ and all $(d+2)$-Schoenberg coef\/f\/icients are positive, then $I f\in\Psi_d$ and all its $d$-Schoenberg coef\/f\/icients are positive;
	\item[(iv)] if $f\in\Psi_d$ and $Df$ is continuous, then $Df\in\Psi_{d+2}$;
	\item[(v)] if $f\in\Psi_d^+$ and $Df$ is continuous, then $Df\in\Psi_{d+2}^+$.	
	\end{enumerate}

Observe that the property of (strict) positive def\/initeness of $f$ is preserved by the operators {\em Mont{\'e}e} $I$ and {\em Descente} $D$.

In this paper, inspired by the work of Beatson and zu Castell \cite{beat-zucatell}, we study positive def\/inite functions on complex unit spheres $\Omega_{2q}$ of $\mathbb C^q$. In particular, we provide walks through dimensions over complex spheres.

Below, we state our main results and we refer to Section \ref{s-backg} for the necessary background.

 We denote the class of positive def\/inite functions on $\Omega_{2q}$ by $\Psi(\Omega_{2q})$. A characterization of such functions was proposed in Menegatto and Peron \cite{P-valdir-pd-esfcompl}: let $ \DD:=\{z\in\mathbb C \colon |z|\leq1\}\subset\mathbb{C}$, when a continuous function $f\colon \DD\to\mathbb{C}$ belongs to $\Psi(\esf)$, an expansion similar to \eqref{eq-exp-scho} exists, namely
 \begin{gather*}
 f(z) = \sum_{m,n\geq0}a_{m,n}^{q-2} R_{m,n}^{q-2}(z), \qquad z \in \DD,
 \end{gather*}(see equation \eqref{eq_defZer} and Theorem~\ref{t-pd-esf}). We will call the coef\/f\/icients $a_{m,n}^{q-2}$ as {\em $(2q)$-complex Schoenberg coefficients}.

In order to make the statements clear, it is convenient to introduce the Descente and Mont{\'e}e operators in the complex context.

Given $f\colon \DD\to\mathbb{C}$, we say that $f$ is dif\/ferentiable if, writing $z=x+iy\in\DD$, $f$ is dif\/ferentiable as a function of $x$ and $y$. Then, we denote by ${\cal D}_x f$ and ${\cal D}_y f$ the partial derivatives with respect to~$x$ and~$y$, respectively, and we def\/ine the Descente operators through the following {\it Wirtinger} derivatives:
\begin{gather}\label{eq-deriv_x_y}
\Dz f =\frac12( {\cal D}_x f - i{\cal D}_y f),
\qquad
\Dcz f = \frac12({\cal D}_x f + i{\cal D}_y f).
\end{gather}
We observe that $f$ might not be complex dif\/ferentiable, actually it is so only when $\Dcz f =0$, and in this case $ \Dz f =f'$, the complex derivative of $f$.

If $f$ admits a $z$-primitive $F$ and a $\ovl{z}$-primitive $G$ in $\DD$, that is, ${\cal D}_z F = {\cal D}_{\ovl{z}}G= f$, then we can def\/ine the Mont\'ee operators ${\cal I}$ and ${\cal \ovl{I}}$ by
\begin{gather*}
{\cal I}(f)(z) := F(z) - F(0) \qquad \text{and} \qquad {\cal \ovl{I}}(f)(z) := G(z) - G(0), \qquad z\in \DD.
\end{gather*}
By def\/inition,
\begin{gather} \label{eq-der-int}
{\cal D}_z ({\cal I} f) = f \qquad \text{and} \qquad {\cal D}_{\ovl{z}}(\ovl{{\cal I}}f) = f.
\end{gather}
Moreover,
\begin{gather*}
{\cal I}({\cal D}_z(f))(z) = f(z)-f(0) \qquad \text{and} \qquad \ovl{{\cal I}}({\cal D}_{\ovl{z}}(f))(z) = f(z)-f(0), \qquad z\in \DD.
\end{gather*}

Our main results are related with walks through dimensions for Descente and Mont{\'e}e ope\-ra\-tors over complex spheres:

\begin{Theorem} \label{t-descente-esf}
Let $f\colon\DD\to\mathbb{C}$ be continuously differentiable.
\begin{enumerate}\itemsep=0pt
\item[$(i)$] If $f$ belongs to the class $\Psi(\Omega_{2q})$, then $\Dz f$, $\Dcz f$ and $\D_xf$ belong to the class $\Psi(\Omega_{2q+2})$.
\item[$(ii)$] If $f$ belongs to the class $\Psi(\Omega_{2q})$ and has all positive $(2q)$-complex Schoenberg coefficients, then $\Dz f$, $\Dcz f$ and $\D_xf$ belong to the class $\Psi^+(\Omega_{2q+2})$.
\end{enumerate}
\end{Theorem}

\begin{Theorem} \label{t-montee-esf}
Let $f\colon \DD\to\mathbb{C}$ be a continuous function admitting a $z$-primitive and a $\ovl{z}$-primitive in $\DD $.
\begin{enumerate}\itemsep=0pt
\item[$(i)$] If $f$ belongs to the class $\Psi(\Omega_{2q+2})$, then there exist real constants $c$ and $C$ such that $c+\I f$ and $C+\Ic f$ belong to the class $\Psi(\Omega_{2q})$.
\item[$(ii)$] If $f$ belongs to the class $\Psi^+(\Omega_{2q+2})$, then there exist real constants $c$ and $C$ such that $c+\I f$ and $C+\Ic f$ belong to the class $\Psi^+(\Omega_{2q})$.
\end{enumerate}
\end{Theorem}

Observe that in Theorem \ref{t-descente-esf}(ii) we assumed the additional condition that all $(2q)$-complex Schoenberg coef\/f\/icients are positive. This condition can be weakened (see Remark~\ref{r-hipotese-relax} below), but not completely removed.

In fact, the following counterexamples show that the Descente operators over complex spheres do not preserve, in general, strict positive def\/initeness, in contrast to the real case of Beatson and zu Castell.
\begin{cex}\label{cex} Let $q\ge 2$ be an integer.
	\begin{enumerate}\itemsep=0pt
		\item[$(i)$]If $f(z) = \sum\limits_{m=0}^\infty a_{m,0}^{q-2}R_{m,0}^{q-2}(z)$, where $\sum\limits_{m=0}^\infty a_{m,0}^{q-2} <\infty$ and $a_{m,0}^{q-2}>0$ for all $m$, then $f\in\Psi^+(\esf)$ and $\Dz f, {\cal D}_xf\in\Psi^+(\esfqq)$ but $\Dcz f \not\in\Psi^+(\esfqq)$.
		\item[$(ii)$] If $f(z) = \sum\limits_{n=0}^\infty a_{0,n}^{q-2}R_{0,n}^{q-2}(z)$, where $\sum\limits_{n=0}^\infty a_{0,n}^{q-2} <\infty$ and $a_{0,n}^{q-2}>0$ for all $n$, then $f\in\Psi^+(\esf)$ and $\Dcz f, {\cal D}_xf\in\Psi^+(\esfqq)$ but $\Dz f\not\in\Psi^+(\esfqq)$.
		\item[$(iii)$] If $f(z) = \sum\limits_{n=0}^\infty a_{0,n}^{q-2}R_{0,n}^{q-2}(z)+ \sum\limits_{m=0}^\infty a_{m,0}^{q-2}R_{m,0}^{q-2}(z)$, where $a_{0,n}^{q-2},a_{m,0}^{q-2}\geq0$ for all $m$, $n$, and
\begin{gather*}
a_{0,n}^{q-2}>0 \Longleftrightarrow n\in 5\mathbb{Z}_+ +4,\\
a_{m,0}^{q-2}>0 \Longleftrightarrow m\in (5\mathbb{Z}_+\setminus\{0\})\cup(5\mathbb{Z}_+ + 2)\cup(5\mathbb{Z}_+ + 3)\cup(5\mathbb{Z}_+ + 4),
\end{gather*}
		then $f\in\Psi^+(\esf)$ but $\Dz f, \Dcz f,\mathcal D_xf \notin\Psi^+(\esfqq)$.
	\end{enumerate}
\end{cex}
	
\begin{Remark}\label{r-hipotese-relax} In the real case, the condition that all $d$-Schoenberg coef\/f\/icients are positive is satisf\/ied by most of the functions in the class $\Psi_d^+$ which appear in applications such as in statistics and geostatistics.
	
In the complex case, among the examples that we provide in Section~\ref{s-familias}, only the exponential function satisf\/ies this condition.
On the other hand, the Akta\c{s}--Ta\c{s}delen--Yavuz, Horn and Lauricella families, satisfy the following simple weaker condition, which is also suf\/f\/icient to obtain the conclusion of Theorem~\ref{t-descente-esf}(ii):
\begin{itemize}\itemsep=0pt
\item {\em if $a_{m,n}^{q-2}$ are the $(2q)$-complex Schoenberg coefficients of $f$, then for some $c,d\in\mathbb N$, the set
\begin{gather*}
\big\{m-n\colon a_{m,n}^{q-2}>0,\, m,n\geq c\big\}
\end{gather*}
contains $(d+\mathbb Z_+)$ or $(-d-\mathbb Z_+)$.}
\end{itemize}

In fact, the weakest possible condition to be used in Theorem \ref{t-descente-esf}(ii) follows from Guella and Menegatto~\cite{jean-men-spd-esf-compl} and reads as follows:
\begin{gather}\label{eq-condit-para-desc}
\big\{m-n\colon a_{m,n}^{q-2}>0,\, m,n\geq1\big\}\cap (N\mathbb Z+j) \neq \varnothing,
\end{gather}
for every $N\geq1$, $j=0,1,\ldots,N-1$. We will prove Theorem~\ref{t-descente-esf} with this last condition, since the previous ones are stronger.
\end{Remark}

This paper is organized as follows: in Section~\ref{s-backg}, we provide the necessary background about positive def\/inite functions on complex spheres and we give a list of parametric families of these functions, which are of interest for both numerical analysis and geostatistical communities. Finally, in Section~\ref{s-proofs}, we obtain all necessary technical lemmas, we give the proofs of Theorems~\ref{t-descente-esf} and~\ref{t-montee-esf}, and we show the Counterexample~\ref{cex}.

\section[The classes $\Psi(\Omega_{2q})$ and $\Psi^+(\Omega_{2q})$: a brief survey]{The classes $\boldsymbol{\Psi(\Omega_{2q})}$ and $\boldsymbol{\Psi^+(\Omega_{2q})}$: a brief survey}\label{s-backg}

This section is largely expository and presents some basic facts and background needed for a~self contained exposition.

For $q$ being a positive integer, we denote by $\esf$ the unit sphere of $\mathbb{C}^q$ and by $B_{2q} := \{z\in \mathbb{C}^q\colon |z|\leq1\}$ the closed disk in $\mathbb{C}^q$. Also, we def\/ine the Pochhammer symbol $ (a)_n := a(a+1)\cdots(a+n-1)$, with $(a)_0:= 1$.

Let $A$ be a nonempty set. A continuous kernel $K\colon A^{2} \rightarrow \mathbb{C}$ is {\em positive definite} if and only if
\begin{gather} \label{posdef}
 \sum_{\mu,\nu=1}^{l}c_{\mu}\ovl{c_{\nu}}K(\xi_{\mu},\xi_{\nu}) \geq 0,
\end{gather}
for all $l \in \mathbb{Z}_+:=\{0,1,2,\ldots\}$, $\{\xi_1, \xi_2,\ldots, \xi_l\} \subset A$ and $\{c_1, c_2, \ldots, c_l\} \subset \mathbb{C}$. If the inequality in~(\ref{posdef}) is strict when at least one $c_\mu$ is nonzero, then $K$ is called {\em strictly positive definite}. For $q$ a strictly positive integer, we def\/ine $A_{q}:= \Omega_2$ when $q=1$ and $A_{q}:= \DD $ for $q>1$. Throughout we shall work with the class $\Psi(\Omega_{2q})$ of continuous functions $f\colon A_{q} \to \mathbb{C}$ such that the kernel $K\colon \Omega_{2q} \times \Omega_{2q} \to \mathbb{C}$ def\/ined as
\begin{gather} \label{isotropy}
 K(\xi,\eta) = f(\la\xi,\eta\ra), \qquad (\xi,\eta) \in \Omega_{2q} \times \Omega_{2q}, \end{gather}
where the symbol $\la\cdot,\cdot\ra$ denotes the usual inner product in $\mathbb C^q$, is positive def\/inite.

Observe that an immediate consequence of the def\/inition is that $f$ satisf\/ies $f(\ovl z)=\ovl{f(z)}$. We shall use the notation~$\Psi^{+}(\Omega_{2q})$ if the kernel $K$ associated to $f$ through (\ref{isotropy}) is strictly positive def\/inite. Positive def\/inite kernels satisfying the identity above are called isotropic. The class~$\Psi(\Omega_{2q})$ is parenthetical to the class $\Psi_d$ introduced by Schoenberg~\cite{scho-42}, and we refer the reader to the recent review in Gneiting \cite{gneiting-2013} for a thorough description of the properties of this class. Further, the class $\Psi_d$ represents the building block for extension to product spaces, and the reader is referred to Berg and Porcu~\cite{porcu-berg} as well as to Guella et al.~\cite{P-jean-men-pdSMxSm} for recent ef\/forts in this direction. The classes $\Psi(\Omega_{2q})$ are nested, with the following inclusion relation being strict:
\begin{gather*}
\Psi(\Omega_4) \supset \Psi(\Omega_6) \supset \cdots \supset \Psi(\Omega_{\infty}),
\end{gather*}
where $\Omega_\infty$ is the unit sphere in the Hilbert space $\ell_2(\mathbb C)$. Analogous relations apply to $\Psi^+(\Omega_{2q})$.

Observe that the class $\Psi(\Omega_2)$ is a dif\/ferent class and it can not be added to the inclusions above (see Menegatto and Peron \cite{P-valdir-pd-esfcompl}). For this reason, in this work we always consider $q\geq2$. Actually the main purpose here is to study the walks through dimensions considering functions in the classes $\Psi(\Omega_{2q})$.

Characterization theorems for the classes $\Psi(\Omega_{2q})$ are available in recent literature, and some ingredients are needed for a detailed exposition. We refer to Boyd and Raychowdhury \cite{boyd}, Dreseler and Hrach \cite{dreseler}, and Koornwinder \cite{koor-II, koor-III} for more information concerning this necessary material.

The {\em disc polynomial} $R_{m,n}^{\alpha}$ of degree $m+n$ in $x$ and $y$ associated to a real number $\alpha>-1$ was introduced by Zernike~\cite{zernike} and Zernike and Brinkman~\cite{zernike-brinkman}, see also Koornwinder~\cite{koor-II}, as the polynomial given by
\begin{gather}\label{eq_defZer}
 R_{m,n}^{\alpha}(z):=r^{|m-n|}e^{i(m-n)\theta}R_{\min\{m,n\}}^{(\alpha,|m-n|)}\big(2r^2 -1\big), \qquad z=re^{i\theta}=x+iy\in \DD,
\end{gather}
where $R_{k}^{(\alpha,\beta)}$ is the usual Jacobi polynomial of degree $k$ associated to the numbers $\alpha,\beta>-1$ and normalized by~$R_{k}^{(\alpha,\beta)}(1)=1$ (see Szeg\H{o} \cite[p.~58]{szego}). Note that the function $R_{m,n}^{\alpha}$ is a~polynomial of degrees~$m$ and~$n$ with respect to the arguments $z$ and $\ovl{z}$, respectively. Moreover it satisf\/ies $R_{m,n}^{\alpha}(\overline z)=\overline {R_{m,n}^{\alpha}(z)}$.

Let ${\rm d}\nu_{\alpha}$ be the positive measure having total mass identically equal to one on $\DD $, and given by
\begin{gather}\label{eq-med-nu}
{\rm d}\nu_{\alpha}(z)=\frac{\alpha+1}{\pi}\big(1-x^2-y^2\big)^{\alpha}{\rm d}x{\rm d}y, \qquad z=x+iy.
\end{gather}
Due to the orthogonality relations for Jacobi polynomials, the set $\{R_{m,n}^{\alpha}\colon 0 \leq m,n<\infty\}$ forms a complete orthogonal system in $L^{2}(\DD,{\rm d}\nu_{\alpha})$
with
\begin{gather}\label{eq-ortog_disc_pol}
\int_{\DD } R_{m,n}^{\alpha}(z)\ovl{R_{k,l}^{\alpha}(z)} {\rm d}\nu_{\alpha}(z) = \frac{1}{h_{m,n}^\alpha}\delta_{m,k}\delta_{n,l},
\end{gather}
where
\begin{gather}\label{eq-const_hmn}
 h_{m,n}^\alpha = \frac{m+n+\a+1}{\a+1}\left(
 \begin{matrix}
 \a +m \\
 \a \end{matrix} \right) \left( \begin{matrix}
 \a+n \\ \a \end{matrix} \right),
\end{gather} and $\delta_{n,l}$ denotes the Kronecker delta. Thus, a function $f\in L^1(\DD,\nu_\a)$, $\a\geq0$, has an expansion in terms of disc polynomials $R_{m,n}^\a$ def\/ined through
\begin{gather}\label{eq-exp-dix-pol}
f(z) \sim \sum_{m,n\geq0}a_{m,n}^\a R_{m,n}^\a(z),
\end{gather}
where
\begin{gather} \label{eq-viado}
a_{m,n}^\a = h_{m,n}^\a \int_{\DD }f(z)\ovl{R_{m,n}^\a(z)}{\rm d}\nu_a(z).
\end{gather}

The Poisson--Szeg\H{o} kernel will be a fundamental tool for the proof of Theorem~\ref{t-pd-esf}(1) below: the characterization of the class~$\Psi(\esf)$. We give here a brief presentation of it, since this kernel will also be used ahead. The Poisson--Szeg\H{o} kernel is def\/ined by
\begin{gather}\label{eq-def-poissson-szego}
{\cal P}_q(r\xi,\eta) := \frac{1}{\sigma_{2q}} \frac{(1-|r\xi|^2)^q}{|1-\la r\xi,\eta\ra|^{2q}}, \qquad r\in[0,1), \quad \xi,\eta\in\esf,
\end{gather}
where $\sigma_{2q}$ is the total surface of $\esf$. Folland \cite{folland} proved that it has an expansion in terms of disc polynomials as
\begin{gather}\label{eq-pois-sze-folland}
{\cal P}_q(r\xi,\eta) =	\sum_{m,n\geq0}\frac{h_{m,n}^{q-2}}{\sigma_{2q} } S_{m,n}^q(r)R_{m,n}^{q-2}(\langle\xi,\eta\rangle), \qquad \xi,\eta\in\esf,\quad r\in[0,1),
\end{gather}
where $S_{m,n}^q(r)\geq 0$, $\lim\limits_{r\to1^-}S_{m,n}^q(r)=1$ and the series converges absolutely and uniformly for $\xi,\eta\in\esf$ and $0 \leq r \leq R$, for each $R <1$.

The Poisson--Szeg\H{o} kernel also appears in the solution of the following Dirichlet problem for the Laplace--Beltrami operator $\Delta_{2q}$ (see Stein~\cite{stein}): given a continuous function $h\colon \esf \rightarrow \mathbb{C}$, there exists a continuous function $u \colon B_{2q} \rightarrow \mathbb{C} $ such that $\Delta_{2q}u=0$ and $u|_{\esf} = h$. The solution $u$ can be computed through
\begin{gather}\label{eq-dirich-probl}
u(z)=\int_{\esf}{\cal P}_q(z,\rho)h(\rho){\rm d}\omega_{2q}(\rho), \qquad z \in B_{2q},
\end{gather}
where ${\rm d}\omega_{2q}$ denotes the rotation-invariant surface element on $\esf$.

In fact, using this, if $\alpha=q-2\geq0$ is an integer and $f$ is a continuous function on~$\DD $, the coef\/f\/icients in the series in~\eqref{eq-exp-dix-pol}, can be written as (see Menegatto and Peron~\cite{P-valdir-pd-esfcompl}):{\samepage
\begin{gather}\label{eq-coef-exp-disc-f-cont}
a_{m,n}^{q-2} = \frac{h_{m,n}^{q-2}}{\sigma_{2q}}\int_{\esf} f(\la\rho,e_1\ra) R_{m,n}^{q-2}(\la e_1,\rho\ra) {\rm d}\omega_{2q}(\rho),
\end{gather}
where $e_1=(1,0,\ldots,0)\in\esf$.}

We give now the representations for the elements of the classes $\Psi(\Omega_{2q})$ and $\Psi^+(\Omega_{2q})$ that were proved by Menegatto and Peron \cite{P-valdir-complexapproach, P-valdir-pd-esfcompl} and Guella and Menegatto~\cite{jean-men-spd-esf-compl}:

\begin{Theorem} \label{t-pd-esf} Let $f\colon \DD \to\mathbb{C}$ be a continuous function. The following assertions are true:
\begin{enumerate}\itemsep=0pt
\item[$(1)$] $f\in\Psi(\esf)$ if, and only if,
\begin{gather} \label{eq-repres-pd-esf}
f(z) = \sum_{m,n\geq0}a_{m,n}^{q-2} R_{m,n}^{q-2}(z), \qquad z \in \DD,
\end{gather}
where $\sum\limits_{m,n\geq0}a_{m,n}^{q-2} <\infty$ and $a_{m,n}^{q-2}\geq0$ for all $(m,n)$;
\item[$(2)$] $f\in \Psi^+(\esf)$ if, and only if, $f\in\Psi(\esf)$ and
\begin{gather}\label{eq-condit-para-spd}
\big\{m-n\colon a_{m,n}^{q-2}>0,\, m,n\geq0\big\}\cap (N\mathbb Z+j) \neq \varnothing,
\end{gather}
for every $N\geq1$, $j=0,1,\ldots,N-1$.
\end{enumerate}
\end{Theorem}

Note that the index $\a=q-2$ of the disc polynomials is related to the sphere $\esf$ and consequently $\a+1=q-1$ is related to $\Omega_{2q+2}$.

The coef\/f\/icients $a_{m,n}^{q-2}$ are the analogue of the $d$-Schoenberg coef\/f\/icients~$a_k^d$ as in Daley and Porcu~\cite{daley-porcu} and Ziegel \cite{ziegel}, referring to the expansion of the members of the Schoenberg class~$\Psi_d$. In analogy, we will call $a_{m,n}^{q-2}$ as {\em $(2q)$-complex Schoenberg coefficients}.

\subsection[Families within the classes $\Psi(\esf)$ and $\Psi^+(\esf)$]{Families within the classes $\boldsymbol{\Psi(\esf)}$ and $\boldsymbol{\Psi^+(\esf)}$} \label{s-familias}

It is well known that there exist many examples of functions in the class $\Psi_d$, some of them widely used in applications (see for example Gneiting \cite{gneiting-2013} and Porcu et al.~\cite{porcu-bev-gent}).

In the literature it is also possible to f\/ind examples of functions that satisfy the conditions in Theorem \ref{t-pd-esf}, or those in Remark \ref{r-hipotese-relax}, and therefore they belong to the classes~$\Psi(\Omega_{2q})$ and~$\Psi^+(\Omega_{2q})$. Some of them, as well as their use in applications, appeared recently, probably originated by the work of W\"unsche~\cite{wunsche}, that deals with disc polynomials: a fundamental tool for studying the functions in these classes. We give below a collection of such functions.

{\bf 1. Disk Polynomials and related families.} The product kernel (Boyd and Raychowdhury~\cite{boyd}),
\begin{gather*}
f_{m,n}(z) = z^m\ovl{z}^n = \sum_{j=0}^{\min\{m,n\}} c_{q,m,n}^j
R_{m-j,n-j}^{q-2}(z), \qquad c_{q,m,n}^j\geq0, \quad z\in \DD,
\end{gather*}
is an element of the class $\Psi(\esf)$, for each $m,n\geq0$.

{\bf 2. Poisson--Szeg\H{o} kernel and related families.} An application of \eqref{eq-def-poissson-szego} and \eqref{eq-pois-sze-folland} shows that
\begin{gather*}
f_r(z) := \frac1{\sigma_{2q}}
\frac{(1-r^2)^q}{|1-rz|^{2q}} = \sum_{m,n\geq0}\frac{h_{m,n}^{q-2}}{\sigma_{2q}}
S_{m,n}^q(r)R_{m,n}^{q-2}(z),\quad \ z\in
\DD,
\end{gather*}
and hence it is a member of the class $\Psi(\esf)$, for each $r\in[0,1)$.

{\bf 3. Exponential function.} The function (Menegatto et al.~\cite{P-valdir-claude-construction})
\begin{gather*}
e^{z+\ovl{z}}=\sum_{m+n=0}^\infty
\frac{(m+1)_{q-2}(n+1)_{q-2}}{(q-2)!}\left(\sum_{j=0}^{\infty}\frac{1}{j!(m+n+q-1)_j}\right)R_{m,n}^{q-2}(z), \qquad z\in \DD,
\end{gather*}
belongs to the class $\Psi^+(\esf)$.

{\bf 4. Akta\c{s}, Ta\c{s}delen and Yavuz family.} The function (Akta{\c{s}} et al.~\cite{aktas})
\begin{gather*}
f_t(z) := \frac{1}{R} \left(\frac2{1-t+R}\right)^{q-2} e^{(2tz)/(1+t+R)}= \sum_{m,n\geq0}(q-1)_n\frac{t^{m+n}}{m!n!} R_{m+n,n}^{q-2}(z), \qquad z\in \DD,
\end{gather*}
where $R:=\big(1-2\big(2|z|^2-1\big)t+t^2\big)^{1/2}$, is a member of $\Psi^+(\esf)$, for each $t\in(0,1)$.

{\bf 5. Horn family.} Let $r$, $R$ be positive integers such that $4r=(R-1)^2$. Horn's function $H_4$ is def\/ined on p.~57 of Srivastava and Manocha~\cite{srivastava} by
\begin{gather*}
H_4(a,b;c,d;x,y) = \sum_{m,n=0}^\infty\frac{(a)_{2m+n}(b)_n}{(c)_m(d)_n} \frac{x^my^n}{m!n!},
\end{gather*}
where $|x|<r$ and $|y|<R$.
An application of Theorem~2.2 in Akta{\c{s}} et al.~\cite{aktas} shows that
\begin{gather*}
f_{t,s,b}(z) := \frac1{(1-s)^{q-1}}H_4\left(q-1,b;q-1,q-1;\frac{s(|z|^2-1)}{(1-s)^2},\frac{t\ovl{z}}{1-s}\right)\\
\hphantom{f_{t,s,b}(z)}{} = \sum_{m,n\geq0}(q+n-1)_m(b)_n\frac{t^ns^m}{m!n!} R_{m,m+n}^{q-2}(z), \qquad z\in \DD.
\end{gather*}
Hence it is a member of $\Psi^+(\esf)$, for each $b$, a positive integer, and $t$, $s$ positive numbers satisfying
\begin{gather*}\label{eq-cond-Horn-t-s}
|s|<1, \qquad \frac{|s|}{(1-s)^2}<r,\qquad \text{and}\qquad \frac{|t|}{1-s}<R.
\end{gather*}

{\bf 6. Lauricella family.} Let $r_1$, $r_2$ and $r_3$ be positive integers such that $r_1r_2=(1-r_2)(r_2-r_3)$. The Lauricella hypergeometric function of three variables~$F_{14}$ (Saran's notation $F_F$ is also used (Saran~\cite{saran})) is def\/ined by (see p.~67 of Srivastava and Manocha~\cite{srivastava})
\begin{gather*}
F_{14}(a_1,a_1,a_1,b_1,b_2,b_1;c_1,c_2,c_2;x_1,x_2,x_3) = \sum_{m,n,p=0}^\infty\frac{(a_1)_{m+n+p}(b_1)_{m+p}(b_2)_n}{(c_1)_m(c_2)_{n+p}} \frac{x_1^mx_2^nx_3^p}{m!n!p!},
\end{gather*}
where $|x_1|<r_1$, $|x_2|<r_2$ and $|x_3|<r_3$. For $t,s\in\mathbb{R}$ such that $|s|<r_1$ and $|t|<r_2$, where $r_1=r_2(1-r_2)$, def\/ine
\begin{gather*}
f_{t,s,b}(z) := F_{14}\big(1,1,1,q-1,b,q-1;q-1,1,1;s\big(|z|^2-1\big),tz,s|z|^2\big), \qquad z\in \DD.
\end{gather*} From Theorem~2.3 in Akta{\c{s}} et al.~\cite{aktas} we get
\begin{gather*}
f_{t,s,b}(z) = \sum_{m,n\geq0}(q-1)_n(b)_m\frac{t^ms^n}{m!n!} R_{m+n,n}^{q-2}(z), \qquad z\in \DD,
\end{gather*}
and hence, $f_{t,s,b}$ is a member of $\Psi^+(\esf)$, for each $b$, a positive integer, and $t$, $s$ positive numbers satisfying the relevant conditions above.

Some comments are in order. Lauricella functions are generalizations of the Gauss hyper\-geo\-metric functions to multiple variables and were introduced by Lauricella in 1893. Recursion formulas and integral representation for Lauricella functions, including $F_{14}$ ($F_F$), have been studied and can be found, for example, in Sahai and Verma \cite{sahai} and Saran \cite{saran1955, saran1957}. In 1873, Schwarz \cite{schwarz} found a list of 15 cases where hypergeometric functions can be expressed algebraically. More precisely, Schwarz gave a list of parameters determining the cases where the hypergeometric dif\/ferential equation has two independent solutions that are algebraic functions. Between 1989 and 2009 several researchers extended this list: to general one-variable hypergeometric functions $_{p+1}F_p$ (Beukers and Heckman \cite{beukers}), the Appell--Lauricella functions~$F_1$ and~$F_D$ (Beazley~Cohen and Wolfart~\cite{wolfart}), the Appell functions~$F_2$ and~$F_4$ (Kato \cite{kato1, kato2}), and the Horn function $G_3$ (Schipper~\cite{schipper}). In 2012, Bod~\cite{bod} extended Schwarz' list to the four classes of Appell--Lauricella functions and the 14 complete Horn functions, including~$H_4$.

\section{Proof of the results}\label{s-proofs}

In this section we f\/irst prove some technical lemmas. Then, we shall be able to give the proof of our main results and to present the counterexamples.

The f\/irst lemma contains recurrence formulas connecting disc polynomials of dif\/ferent indexes and degrees. They are obtained from equation~(5.5) in Aharmim et al.~\cite{aharmim} and the following properties of the disc polynomials
\begin{gather*}
\ovl{R_{m,n}^\a(z)}=R_{n,m}^\a(z), \qquad \ovl{{\cal D}_{z} R_{m,n}^\a(z)} = {\cal D}_{\overline{z}} R_{n,m}^\a(z),
 \qquad \a>-1, \quad m,n\geq0,\quad z\in \DD.
\end{gather*}
We observe that the normalization adopted in Aharmim et al.~\cite{aharmim} for the disc polynomials is dif\/ferent from the one we use here.

\begin{Lemma} \label{l-rel-alpha+1_alpha} Let $m$, $n$ be non negative integers and $\a>-1$ be a real number. Then, for any $z\in \DD $, we have
\begin{gather}\label{eq-rel-alpha+1_alpha}
(\a+1)R_{m,n+1}^\a(z) = (\a+1)\ovl{z}R_{m,n}^{\a+1}(z)-\big(1-|z|^2\big){\cal D}_{z} R_{m,n}^{\a+1}(z),
\end{gather}
and
\begin{gather}\label{eq-rel-alpha+1_alpha-conj}
(\a+1)R_{n+1,m}^\a(z) = (\a+1){z}R_{n,m}^{\a+1}(z)-\big(1-|z|^2\big){\cal D}_{\overline{z}} R_{n,m}^{\a+1}(z).
\end{gather}
\end{Lemma}

Below we prove an important technical result, that connects the expansion of a conti\-nuously dif\/ferentiable function $f$ in terms of the disc polynomials $R_{m,n}^\a$ with the expansion of its derivatives in terms of the disc polynomials~$R_{m,n}^{\a+1}$.

Since $R_{m,n}^{q-2}$ belongs to $\Psi(\esf)$ when $q\geq2$ is an integer, this connection will be the main ingredient in order to obtain preservation of positive def\/initeness for the Descente operators, when walks through dimensions over complex spheres are provided.

\begin{Lemma} \label{l-rel-coef-f-Dzf}Let $f\colon\DD\to\mathbb{C}$ be continuously differentiable and let $\a>-1$ be a real number. Consider the expansion of $f$ in terms of the disc polynomials $R_{m,n}^\a$ and the expansions of $\Dz f$ and $\Dcz f$ in terms of the disc polynomials $R_{m,n}^{\a+1}$
\begin{gather*}
f(z) \sim \sum_{m,n=0}^\infty a_{m,n}^\a R_{m,n}^\a(z), \qquad z \in \DD,\\
{\cal D}_{z}f(z) \sim \sum_{m,n=0}^\infty b_{m,n}^{\a+1} R_{m,n}^{\a+1}(z) \qquad \text{and} \qquad
{\cal D}_{\overline{z}}f(z) \sim \sum_{m,n=0}^\infty \tilde b_{m,n}^{\a+1} R_{m,n}^{\a+1}(z), \qquad z \in \DD.
\end{gather*}
Then,
\begin{gather*}
b_{m,n}^{\a+1} = \frac{(m+1)(n+\a+1)}{(\a+1)} a_{m+1,n}^\a, \qquad m,n\geq0,
\end{gather*}
and
\begin{gather*}
\tilde b_{m,n}^{\a+1} = \frac{(n+1)(m+\a+1)}{(\a+1)} a_{m,n+1}^\a, \qquad m,n\geq0.
\end{gather*}
\end{Lemma}

It is worth noting that this result is not surprising if we consider the identities obtained in Koornwinder~\cite{koor-III}: for $\a>-1$,
\begin{gather}\label{eq-derv-em-rel-z}
\Dz R_{m,n}^\a = c_\a(m,n)R_{m-1,n}^{\a+1} \qquad \text{and} \qquad \Dcz R_{m,n}^\a = c_\a(n,m)R_{m,n-1}^{\a+1},
\end{gather}
where $c_\a(m,n):= (m(n+\a+1))/(\a+1)$. These are, in the complex case, the analogue of the identities for the derivative of the Gegenbauer polynomials (see Szeg\H{o} \cite[equation~(4.7.14)]{szego}).

Actually, Lemma \ref{l-rel-coef-f-Dzf} shows that the coef\/f\/icients in the expansions are linked as if the series could be derived term by term.

\begin{proof}[Proof of Lemma \ref{l-rel-coef-f-Dzf}] The coef\/f\/icients $b_{m,n}^{\a+1}$ are given by the formula
\begin{gather*}
b_{m,n}^{\a+1} = h_{m,n}^{\a+1} \int_{\DD }{\cal D}_{z}f(z)\ovl{R_{m,n}^{\a+1}(z)}{\rm d}\nu_{\a+1}(z),
\end{gather*}
where the constants $h_{m,n}^{\a+1}$ are given in \eqref{eq-const_hmn}. Def\/ine
\begin{gather*}
I := \int_{\DD } {\cal D}_{z}f(z){R_{n,m}^{\a+1}(z)}{\rm d}\nu_{\a+1}(z)
= \frac{\a+2}{\pi} \int_{\DD }{\cal D}_{z}f(z){R_{n,m}^{\a+1}(z)}\big(1-x^2-y^2\big)^{\a+1}{\rm d}x{\rm d}y.
\end{gather*}
Integration by parts and direct inspection shows that
\begin{gather*}
I = \frac{\a+2}{\pi} \bigg\{\int_{\DD }{\cal D}_{z}\big[f(z){R_{n,m}^{\a+1}(z)}\big(1-|z|^2\big)^{\a+1}\big]{\rm d}x{\rm d}y \\
\hphantom{I = \frac{\a+2}{\pi} \bigg\{}{} -\int_{\DD } f(z) {\cal D}_{z}\big[{R_{n,m}^{\a+1}(z)}\big(1-|z|^2\big)^{\a+1}\big]{\rm d}x{\rm d}y \bigg\}.
\end{gather*}
Using Green's theorem and \eqref{eq-deriv_x_y} we have
\begin{gather*}\label{eq-relation_int_B-int_front_conj_z}
\int_{\ccir} g(z) {\rm d}\ovl{z} = - 2i\int_{\DD } {\cal D}_{z}(g)(z) {\rm d}x{\rm d}y,
\end{gather*}
for any continuously dif\/ferentiable function $g$. Thus,
\begin{gather*}
 I = \frac{\a+2}{\pi} \left\{ \frac{i}2\int_{\ccir} \! f(z){R_{n,m}^{\a+1}(z)}\big(1-|z|^2\big)^{\a+1}{\rm d}\ovl{z} - \!\int_{\DD } f(z) {\cal D}_{z}\left[{R_{n,m}^{\a+1}(z)}\big(1-|z|^2\big)^{\a+1}\right]{\rm d}x{\rm d}y \right\}\\
\hphantom{I}{} = - \frac{\a+2}{\pi} \int_{\DD } f(z) {\cal D}_{z}\left[{R_{n,m}^{\a+1}(z)}\big(1-|z|^2\big)^{\a+1}\right]{\rm d}x{\rm d}y.
\end{gather*}
Now, by noting that
\begin{gather*} {\cal D}_{z}\big[{R_{n,m}^{\a+1}(z)}\big(1-|z|^2\big)^{\a+1}\big] =
{\cal D}_{z}{R_{n,m}^{\a+1}(z)}\big(1-|z|^2\big)^{\a+1} - (\a+1)\big(1-|z|^2\big)^{\a}\ovl{z}{R_{n,m}^{\a+1}(z)},
\end{gather*}
we get
\begin{gather*}
I =\frac{\a+2}{\pi} \int_{\DD } f(z)\big(1-|z|^2\big)^{\a}\left[(\a+1)\ovl{z}R_{n,m}^{\a+1}(z) - \big(1-|z|^2\big){\cal D}_{z}{R_{n,m}^{\a+1}(z)}\right] {\rm d}x{\rm d}y.
\end{gather*}
Hence, using Lemma \ref{l-rel-alpha+1_alpha}, we have
\begin{gather*}
I =\frac{\a+2}{\pi} \int_{\DD } f(z)\big(1-|z|^2\big)^{\a}(\a+1)R_{n,m+1}^{\a}(z) {\rm d}x{\rm d}y
=(\a+2) \int_{\DD } f(z)\ovl{R_{m+1,n}^{\a}(z)} {\rm d}\nu_\a(z).
\end{gather*}
Thus,
\begin{gather*}
b_{m,n}^{\a+1} = h_{m,n}^{\a+1}I = (\a+2)h_{m,n}^{\a+1}\frac{1}{h_{m+1,n}^\a} a_{m+1,n}^\a.
\end{gather*}
Replacing the values of $h_{m,n}^{\a+1}$ and $h_{m+1,n}^\a$ given in equation~\eqref{eq-const_hmn}, we obtain
\begin{gather*}
b_{m,n}^{\a+1} = (\a+2)\frac{(m+1)(\a+n+1)}{(\a+2)(\a+1)} a_{m+1,n}^\a = \frac{(m+1)(\a+n+1)}{(\a+1)} a_{m+1,n}^\a.
\end{gather*}
The proof for the case of the operator $\Dcz $ is analogous observing that
\begin{gather*}\label{eq-relation_int_B-int_front_z}
\int_{\ccir} g(z) {\rm d}z = 2i\int_{\DD } {\cal D}_{\overline{z}}(g)(z) {\rm d}x{\rm d}y.\tag*{\qed}
\end{gather*}
\renewcommand{\qed}{}
\end{proof}

The last technical lemma gives a condition for the expansion of a continuous function in terms of the disc polynomials to be uniformly convergent.

\begin{Lemma} \label{l-conv-ser-coeff} Let $g\colon \DD \to\mathbb{C}$ be a continuous function and consider its expansion
\begin{gather}\label{eq-serie_disc_pol}
g(z) \sim \sum_{m,n\geq0} d_{m,n}^{q-2}R_{m,n}^{q-2}(z),\qquad z\in \DD,
\end{gather}
where
$d_{m,n}^{q-2}$ are given as in \eqref{eq-coef-exp-disc-f-cont}. If $d_{m,n}^{q-2}\geq0$ for all $m,n\geq0$, then
$\sum\limits_{m,n\geq0} d_{m,n}^{q-2}<\infty$. In particular, the series in~\eqref{eq-serie_disc_pol} converges uniformly in $\DD $.
\end{Lemma}

\begin{proof} The argument is similar to the one used in the proof of Theorem~4.1 in Menegatto and Peron \cite{P-valdir-pd-esfcompl}. Given $\xi\in\esf$, consider the continuous function $h(\rho):=g(\la\rho,\xi\ra)$, $\rho\in\esf$. By equation~\eqref{eq-dirich-probl}, the solution of the Dirichlet problem $\Delta_{2q}u=0$ in the interior of $B_{2q}$ with boundary condition $h$, evaluated on the segment $r\xi$, $r\in[0,1)$, is
 \begin{gather*}
u(r\xi) = \int_{\esf}{\cal	P}_{q}(r\xi,\rho)g(\la\rho,\xi\ra){\rm d}\omega_{2q}(\rho) = \sum_{m,n\geq0} S_{m,n}^{q}(r) d_{m,n}^{q-2},
 \end{gather*}
where the last equality is obtained from \eqref{eq-pois-sze-folland}, \eqref{eq-coef-exp-disc-f-cont}.

Since $u$ is continuous up to the boundary and coincides with $h$ on $\esf$, we obtain
\begin{gather*}
\lim_{r\to1^-}\sum_{m,n\geq0} d_{m,n}^{q-2}S_{m,n}^{q}(r) = \lim_{r\to1^-} u(r\xi) = u(\xi)= g(\la \xi,\xi\ra) = g(1).
\end{gather*}
Now, note that
 \begin{gather*}
 0 \leq \sum_{m=0}^k \sum_{n=0}^l d_{m,n}^{q-2}S_{m,n}^{q}(r)\leq \sum_{m,n\geq0} d_{m,n}^{q-2}S_{m,n}^{q}(r), \qquad 0 \leq r < 1.
 \end{gather*}
 Letting $r\to1^-$, we get
\begin{gather*} \label{eq-series-ineq}
0\leq s_{k,l}:= \sum_{m=0}^k \sum_{n=0}^l d_{m,n}^{q-2} \leq
\lim_{r\to1^-}\sum_{m,n\geq0} d_{m,n}^{q-2}S_{m,n}^{q}(r) = g(1), \qquad k,l\in\mathbb{Z}_+.
\end{gather*}
Hence, the sequence $\{s_{k,l}\}_{k,l\in\mathbb{Z}_+}$ is bounded and increasing. Thus, the series $\sum\limits_{m,n\geq0} d_{m,n}^{q-2}$ is convergent.
Using the fact that $|R_{m,n}^{q-2}(z)|\leq1$ for all $z\in \DD $ and using the Weierstrass M-Test, the proof is completed.
\end{proof}

At this point, we are able to prove our main results.

\begin{proof}[Proof of Theorem \ref{t-descente-esf}.] Let $f$ be a function in the class $\Psi(\esf)$. Then, by Theorem~\ref{t-pd-esf}(1),
\begin{gather*}
f(z) = \sum_{m,n\geq0}a_{m,n}^\a R_{m,n}^\a(z), \qquad z \in \DD,
\end{gather*}
where $\a=q-2$, $a_{m,n}^\a\geq0$, for all $m,n\geq0$, and $\sum\limits_{m,n\geq0}a_{m,n}^\a<\infty$. Consider the expansion in terms of disc polynomials of $\Dz f$:
\begin{gather*}
\Dz f(z) \sim \sum_{m,n\geq0}b_{m,n}^{\a+1}R_{m,n}^{\a+1}(z),\qquad z\in \DD.
\end{gather*}

By Lemma \ref{l-rel-coef-f-Dzf} and equation~\eqref{eq-derv-em-rel-z},
\begin{gather}\label{eq-coef-Dz}
b_{m,n}^{\a+1} = c_\a(m+1,n)a_{m+1,n}^\a, \qquad m,n\geq0.
\end{gather}

Roughly speaking, \eqref{eq-coef-Dz} means that the coef\/f\/icients $\big\{b_{m,n}^{\a+1}\big\}$ are obtained from the $\{a_{m,n}^\a\}$ by suppressing the~$a_{0,n}^\a$, translating in the f\/irst index and multiplying by the positive constants $\{c_\a(m+1,n)\}$.

Then, by equation \eqref{eq-derv-em-rel-z}, we have
\begin{gather*}
\sum_{m,n\geq0}a_{m,n}^\a\Dz R_{m,n}^{\a}(z) = \sum_{m\geq-1}\sum_{n\geq0}a_{m+1,n}^\a\Dz R_{m+1,n}^{\a}(z)
 = \sum_{m,n\geq0}b_{m,n}^{\a+1} R_{m,n}^{\a+1}(z).
\end{gather*}

Now, since $c_\a(m+1,n)$ are positive constants, we have that $b_{m,n}^{\a+1}\geq0$ for all $m,n\geq0$. By Lemma~\ref{l-conv-ser-coeff}, the series
\begin{gather*}
\sum_{m,n\geq0}a_{m,n}^\a\Dz R_{m,n}^{\a}(1) = \sum_{m,n\geq0}b_{m,n}^{\a+1}
\end{gather*}
is convergent and the series $\sum\limits_{m,n\geq0}b_{m,n}^{\a+1}R_{m,n}^{\a+1}(z)$ converges uniformly in $\DD $. It follows, by term by term dif\/ferentiation, that
\begin{gather*}
\Dz f(z) = \sum_{m,n\geq0}b_{m,n}^{\a+1}R_{m,n}^{\a+1}(z).
\end{gather*}

Hence, by Theorem \ref{t-pd-esf}(1), $\Dz f$ belongs to the class $\Psi(\Omega_{2q+2})$.

Similarly, we can conclude the same for the operator $\Dcz $.

For the item (ii), observe that, as a consequence of \eqref{eq-coef-Dz}, if the $(2q)$-complex Schoenberg coef\/f\/icients $a_{m,n}^\alpha$ of~$f$ satisfy \eqref{eq-condit-para-desc}, then the $(2q+2)$-complex Schoenberg coef\/f\/icients $b_{m,n}^{\alpha+1}$ of~$\Dz f$ (and similarly for~$\Dcz f$) satisfy~\eqref{eq-condit-para-spd}. Actually, the condition $m,n\geq1$ in the set considered in~\eqref{eq-condit-para-desc} guarantees that the intersections with the arithmetic progressions in $\mathbb Z$ do not depend on the coef\/f\/icients~$a_{m,0}^{\alpha}$ or~$a_{0,n}^{\alpha}$, which are suppressed by the Descente operators.

The results for $\D_xf$ follow immediately by \eqref{eq-deriv_x_y}.
\end{proof}

\begin{proof}[Proof of Theorem \ref{t-montee-esf}] Suppose that $f$ belongs to the class $\Psi(\Omega_{2q+2})$. By Theorem~\ref{t-pd-esf},
\begin{gather*}
f(z) = \sum_{m,n\geq0}a_{m,n}^{\a+1}R_{m,n}^{\a+1}(z), \qquad z \in \DD,
\end{gather*}
where $\alpha=q-2$ and $a_{m,n}^{\a+1}\geq0$ for all $m,n\geq0$ and $\sum\limits_{m,n\geq0}a_{m,n}^{\a+1}<\infty$. By equation~\eqref{eq-derv-em-rel-z}, we have
\begin{gather}\label{eq-int-disc-pol}
\I(R_{m-1,n}^{\a+1})(z) = \frac{1}{c_\a(m,n)} \big(R_{m,n}^\a(z) - R_{m,n}^\a(0) \big),
\end{gather}
where $R_{n,n}^\a(0) = (-1)^nn!\a!/(n+\a)!$ and $R_{m,n}^\a(0) =0$, $m\neq n$ (W\"unsche \cite[equation~(2.9)]{wunsche}). Thus consider
\begin{gather} \label{eq-serie-integral}
F(z):= \sum_{m,n\geq0}a_{m,n}^{\a+1}\I (R_{m,n}^{\a+1})(z) = \sum_{m\geq1}\sum_{n\geq0} \frac{a_{m-1,n}^{\a+1}}{c_\a(m,n)} \left(R_{m,n}^\a(z) - R_{m,n}^\a(0) \right), \qquad z\in \DD.\!\!\!
\end{gather}

Since $c_\a(m,n)\geq1$ for all $m\geq1$, $n\geq0$ and $|R_{m,n}^\a(0)|\leq1$, for all $m$, $n$, we have that the series
\begin{gather*}
\sum_{m\geq1}\sum_{n\geq0} \frac{a_{m-1,n}^{\a+1}}{c_\a(m,n)} \qquad \text{and} \qquad c: = \sum_{m\geq1}\sum_{n\geq0} \frac{a_{m-1,n}^{\a+1}}{c_\a(m,n)} R_{m,n}^\a(0)
\end{gather*}
are convergent. Furthermore, since
\begin{gather*}
\left|\frac{a_{m-1,n}^{\a+1}}{c_\a(m,n)}\left(R_{m,n}^\a(z) - R_{m,n}^\a(0)\right)\right|\leq 2\frac{a_{m-1,n}^{\a+1}}{c_\a(m,n)}, \qquad m,n\geq0, \quad z\in \DD,
\end{gather*}
the series in \eqref{eq-serie-integral} converges uniformly in~$\DD $. On the other hand, by applying the derivation operator~$\D_z$ term by term in~\eqref{eq-serie-integral}, one obtains the uniformly convergent series of~$f$. Then $F$ is a~$z$-primitive of~$f$. Since $F(0)$=0, we conclude that~\eqref{eq-serie-integral} converges to~$\I(f)(z)$.

We can now write
\begin{gather*}
\I (f)(z) = \sum_{m,n\geq0} b_{m,n}^\a R_{m,n}^\a(z),
\end{gather*}
where
\begin{gather}
b_{0,0}^\a := - c;\qquad b_{0,n}^\a := 0, \qquad n\geq1;\nonumber\\
b_{m,n}^\a := \frac{a_{m-1,n}^{\a+1}}{c_\a(m,n)}, \qquad m\geq1,\quad n\geq0; \label{eq-coef-If}
\end{gather}
and $\sum\limits_{m,n\geq0}b_{m,n}^\a<\infty$. Now we can write
\begin{gather*}
c+\I f (z) = \sum_{m,n\geq0} \widehat{b}_{m,n}^\a R_{m,n}^\a(z),
\end{gather*}
where
\begin{gather} \label{eq-coef-If2}
\widehat{b}_{0,0}^\a:=0; \qquad \widehat{b}_{0,n}^\a:=b_{0,n}^\a, \qquad n\geq1\qquad \text{and} \qquad
\widehat{b}_{m,n}^\a:=b_{m,n}^\a, \qquad m\geq1,\quad n\geq0
\end{gather}
are nonnegative constants and $\sum\limits_{m,n\geq0}\widehat{b}_{m,n}^\a<\infty$.

Equations \eqref{eq-coef-If} and \eqref{eq-coef-If2} mean that the coef\/f\/icients $\big\{\widehat b_{m,n}^{\a}\big\}$ are obtained from the $\{a_{m,n}^{\a+1}\}$, by translating in the f\/irst index, adding the new coef\/f\/icients $\widehat b_{0,n}^\a=0$, and dividing by the positive constants $\{c_\a(m,n)\}$.

Hence, applying Theorem \ref{t-pd-esf}(1) again, we have that $c+\I f$ belongs to the class $\Psi(\esf)$.

For the item (ii), it is enough to observe that the $(2q+2)$-complex Schoenberg coef\/f\/icients $a_{m,n}^{\alpha+1}$ of $f$ satisfy~\eqref{eq-condit-para-spd} by the assumption $f\in\Psi^+(\esfqq)$, then, as a consequence of (\ref{eq-coef-If}),~(\ref{eq-coef-If2}), also the $(2q)$-complex Schoenberg coef\/f\/icients $\widehat b_{m,n}^{\alpha}$ of $c+\I f$ satisfy \eqref{eq-condit-para-spd}, implying $c+\I f\in\Psi^+(\esf)$.

For the operator $\Ic $, one uses the relation
\begin{gather*} \label{eq-int-disc-pol-conj}
\Ic(R_{m,n-1}^{\a+1})(z) = \frac{1}{c_\a(n,m)} \big(R_{m,n}^\a(z) - R_{m,n}^\a(0) \big),
\end{gather*}
and follows the same arguments. In fact, the $(2q)$-complex Schoenberg coef\/f\/icients of $C+\Ic f$ are given by
\begin{gather*}
\check b_{0,0}^\a := C - \sum_{\mu\geq1}\sum_{\nu\geq0} \frac{a_{\mu,\nu-1}^{\a+1}}{c_\a(\nu,\mu)} R_{\mu,\nu}^\a(0); \qquad
\check b_{m,0}^\a := 0, \qquad m\geq1;\\
\check b_{m,n}^\a := \frac{a_{m,n-1}^{\a+1}}{c_\a(m,n)}, \qquad m\geq0, \quad n\geq1.\tag*{\qed}
\end{gather*}\renewcommand{\qed}{}
\end{proof}

\begin{proof}[Proof of Counterexample~\ref{cex}.] Let us denote by $a_{m,n}^{q-2}(g)$ the $(2q)$-complex Schoenberg coef\/f\/i\-cients of a positive def\/inite function $g$. Theorem \ref{t-pd-esf}(2) is required.

(i) For a function $f$ as in the statement, we have ${\cal D}_xf=\Dz f$ and
\begin{gather*}
\big\{m-n\colon a_{m,n}^{q-1}(\Dz f)>0\big\} = \big\{m-n\colon a_{m,n}^{q-2}(f)>0\big\} = \mathbb{Z}_+.
\end{gather*}
Hence the above set intercepts every arithmetic progression in $\mathbb Z$, that is $f\in\Psi^+(\esf)$ and $\Dz f, D_xf\in\Psi^+(\esfqq)$. However,
$\Dcz f\equiv 0$, so that $\Dcz f\not\in\Psi^+(\esfqq)$.

(ii) Analogous to (i).

(iii) For a function $f$ as in the statement, we have
\begin{gather*}\big\{m-n\colon a_{m,n}^{q-2}(f)>0, \, m,n\geq0\big\}=\left(\bigcup_{j=2}^55\mathbb Z_++j\right)\cup(-5\mathbb Z_+-4),\end{gather*} which intercepts every arithmetic progression in $\mathbb Z$ and then $f\in\Psi^+(\esf)$. However
\begin{gather*}
\big\{m-n\colon a_{m,n}^{q-1}(\mathcal D_z f)>0, \, m,n\geq0\big\} = \mathbb Z_+\setminus 5\mathbb{Z}
\end{gather*}
and
\begin{gather*}
\big\{m-n\colon a_{m,n}^{q-1}(\Dcz f)>0, \, m,n\geq0\big\} = -5\mathbb{Z}_+-3,
\end{gather*}
that is, $\Dz f, \Dcz f\not\in\Psi^+(\esfqq)$. To see that $\mathcal D_xf\not\in\Psi(\esfqq)$, note that $\{m-n\colon a_{m,n}^{q-1}(\mathcal D_x f)>0$, $m,n\geq0\}$ is the union of the previous two sets, so it does not intersect the progression $5\mathbb Z$.
\end{proof}

\subsection*{Acknowledgement}
The authors gratefully thank the anonymous referees for the constructive comments and recommendations which helped to greatly improve the paper.
Eugenio Massa was supported by grant $\#$2014/25398-0, S\~ao Paulo Research Foundation (FAPESP) and grant \mbox{$\#$308354/2014-1}, CNPq/Brazil.
Ana P.~Peron was supported by grants $\#$2016/03015-7 and $\#$2014/25796-5, S\~ao Paulo Research Foundation (FAPESP).
Emilio Porcu was supported by grant FONDECYT \#1170290 from the Chilean government.

\pdfbookmark[1]{References}{ref}
\LastPageEnding

\end{document}